\documentclass[12pt]{article}
\usepackage{amsmath,amssymb,amsthm}

\newtheorem{thm}{Theorem}
\newtheorem{crl}{Corollary}
\newtheorem{cnj}{Conjecture}

\newtheorem{lmm}{Lemma}
\newtheorem{rmk}{Remark}
\newtheorem{dfn}{Definition}

\makeatletter
\newcommand{\subjclass}[2][2010]{%
  \let\@oldtitle\@title%
  \gdef\@title{\@oldtitle\footnotetext{#1 \emph{Mathematics subject classification(s).} #2}}%
}
\newcommand{\keywords}[1]{%
  \let\@@oldtitle\@title%
  \gdef\@title{\@@oldtitle\footnotetext{\emph{Key words and phrases.} #1.}}%
}
\makeatother

\title{Fermat curves and the reciprocity law on cyclotomic units}
\author{Tomokazu Kashio\thanks{Tokyo University of Science, \texttt{kashio\_tomokazu@ma.noda.tus.ac.jp}}}

\begin{document}
\maketitle

\begin{abstract}
We define a ``period ring-valued beta function'' and give a reciprocity law on its special values.
The proof is based on some results of Rohrlich and Coleman concerning Fermat curves. 
We also have the following application.
Stark's conjecture implies that the exponential of the derivatives at $s=0$ of partial zeta functions are algebraic numbers which satisfy a reciprocity law
under certain conditions.
It follows from Euler's formulas and properties of cyclotomic units when the base field is the rational number field.
In this paper, we provide an alternative (and partial) proof by using the reciprocity law on the period ring-valued beta function. 
In other words, the reciprocity law given in this paper is a refinement of the reciprocity law on cyclotomic units.
\end{abstract}

\section{Introduction}

We first recall the rank $1$ abelian Stark conjecture. 
Let $K/F$ be an abelian extension of number fields with $G:=\mathrm{Gal}(K/F)$.
Let $S$ be a finite set of places of $F$ containing all infinite places and all ramifying places. 
Then the partial zeta function associated to $S,\sigma \in G$ is defined by
\begin{equation*}
\zeta_S(s,\sigma):=\sum_{\mathfrak a \subset \mathcal O_F,(\mathfrak a,S)=1,\left(\frac{K/F}{\mathfrak a}\right)=\sigma}N\mathfrak a ^{-s}.
\end{equation*}
Here $\mathfrak a$ runs over all integral ideals of $F$ relatively prime to all primes in $S$ 
whose image under the Artin symbol $(\tfrac{K/F}{* })$ is equal to $\sigma$.
This series converges when $\mathrm{Re}(s)>1$ and has a meromorphic continuation to the whole complex $s$-plane.

\begin{cnj}[The rank 1 abelian Stark conjecture]
Assume that 
\begin{center}
{\rm(Asmp)} \quad $S$ contains a place $v$ which splits completely in $K/F$ and $|S|\geq 2$.
\end{center}
We choose a place $w$ of $K$ lying above the splitting place $v$ and put $e_K$ to be the number of roots of unity contained in $K$.
Then there exists an element $\varepsilon$, which is called a Stark unit, satisfying the following conditions:
\begin{itemize}
\item[{\rm(Alg)}] $\varepsilon \in K^\times$.
\item[{\rm(Unit)}] If $|S|>2$, then $\varepsilon$ is a $\{v\}$-unit. 
Otherwise, we can write $S=:\{v,v'\}$ and $\varepsilon$ is an $S$-unit satisfying that $|\varepsilon|_{w'}$ is a constant for all places $w' $ of $K$ 
lying above $v'$.
\item[{\rm(Rec)}] $\log |\varepsilon^\sigma|_w=-e_K \zeta'_S(0,\sigma)$ for all $\sigma \in G$.
\item[{\rm(Abel)}] $K\left(\varepsilon^{\frac{1}{e_K}}\right)/k$ is an abelian extension. 
\end{itemize}
\end{cnj}

We note that {\rm(Asmp)} implies $\zeta_S(0,\sigma)=0$ for all $\sigma \in G$.
We consider the case when $v$ is a real place in this paper.
Automatically we may assume that $F$ is totally real since it is known that Conjecture 1 holds true if $S$ contains two places splitting completely in $K/F$.
Namely, we assume that 
\begin{center}
$F$ is a totally real field and $v,w$ are real places.
\end{center}
Then we can write the Stark unit explicitly as $\varepsilon=\exp(-2\zeta_S'(0,\sigma))$ (assuming it exists).
Here $e_K=2$ since $K$ is not totally complex and we regard $K$ as a subfield of $\mathbb R$ by the place $w$.
From now on, we focus on two conditions (Alg),(Rec) of Stark's conjecture when $v$ is a real place.
These are equivalent to the following (Alg$'$),(Rec$'$).

\begin{cnj}[A part of Stark's conjecture] \label{wsc}
Let $F$ be totally real, $K$ its finite abelian extension, and $G:=\mathrm{Gal}(K/F)$. 
Assume that $K\neq \mathbb Q$ and that 
\begin{center}
there exists an embedding $K \hookrightarrow \mathbb R$.
\end{center}
We regard $K$ as a subfield of $\mathbb R$ by this embedding. Then we have 
\begin{itemize}
\item[{\rm(Alg$'$)}] $u_F(\sigma):=\exp(-2\zeta'(0,\sigma)) \in \overline{\mathbb Q}^\times$ for $\sigma \in G$.
\item[{\rm(Rec$'$)}] $\tau(u_F(\sigma))=u_F(\tau\sigma)$ for $\sigma\in G$, $\tau \in G_F:=\mathrm{Gal}(\overline{\mathbb Q}/F)$.
\end{itemize}
Here taking the minimal $S_0:=\{$ infinite places $\} \cup \{$ ramifying places $\}$, 
we drop the symbol $S_0$ from partial zeta functions $\zeta(s,\sigma):=\zeta_{S_0}(s,\sigma)$.  
\end{cnj}

In the case of $F=\mathbb Q$, Conjecture \ref{wsc} (and Conjecture 1 also) follow from Euler's formulas and properties of cyclotomic units. 
We provide a sketch of the proof: When $F=\mathbb Q$, the case of $K=\mathbb Q(\zeta_m)^+:=\mathbb Q(\zeta_m+\zeta_m^{-1})$ 
with $\zeta_m:=\exp(\frac{2\pi\sqrt{-1}}{m})$, $3\leq m \in \mathbb N$ is essential. 
Any element in $G=\mathrm{Gal}(\mathbb Q(\zeta_m)^+/\mathbb Q)$ can be expressed as  
$[\sigma_{\pm\frac{a}{m}} \colon \zeta_m+\zeta_m^{-1} \mapsto \zeta_m^a+\zeta_m^{-a}]$ with $0<a<\frac{m}{2}$, $(a,m)=1$.
Then we have
\begin{equation} \label{suoverq}
u_{\mathbb Q}(\sigma_{\pm\frac{a}{m}})=\left(\frac{2\pi}{\Gamma(\frac{a}{m})\Gamma(\frac{m-a}{m})}\right)^2.
\end{equation}
We can derive the expression (\ref{suoverq}) from the following formulas (\ref{hl}): 
Let $\zeta(s,v,z):=\sum_{n=0}^{\infty}(z+vn)^{-s}$ be the Hurwitz zeta function ($z,v>0$). Then we have
\begin{equation} \label{hl}
\begin{split}
\zeta(s,\sigma_{\pm\frac{a}{m}})&=\zeta(s,m,a)+\zeta(s,m,m-a), \\
\zeta(0,m,a)&=\frac{1}{2}-\frac{a}{m}, \\
\exp\left(\zeta'(0,m,a)\right)&=\Gamma(\tfrac{a}{m})(2\pi)^{-\frac{1}{2}}m^{-\zeta(0,m,a)}.
\end{split}
\end{equation}
Furthermore, we see that 
\begin{equation} \label{suitocu}
u_{\mathbb Q}(\sigma_{\pm\frac{a}{m}})=\left(2\sin(\tfrac{a}{m}\pi)\right)^2=2-\left(\zeta_m^a+\zeta_m^{-a}\right)
\end{equation}
by using Euler's formulas:
\begin{equation} \label{erf}
\Gamma(z)\Gamma(1-z)=\frac{\pi}{\sin (\pi z)},
\end{equation}
\begin{equation} \label{ef}
\sin(z)=\frac{\exp(z\sqrt{-1})-\exp(-z\sqrt{-1})}{2\sqrt{-1}}.
\end{equation}
Then Conjecture \ref{wsc} in the case of $F=\mathbb Q$, $K=\mathbb Q(\zeta_m)^+$ follows immediately from (\ref{suitocu}). 

For any totally real field $F$, Shintani expressed the special values $\zeta(0,\sigma)$ and the derivative values $\zeta'(0,\sigma)$
in terms of Bernoulli polynomials and Barnes' multiple gamma functions, with some correction terms.
Shintani's formulas are generalizations of formulas (\ref{hl}). (For detail, see \cite{Shin} or \cite[Theorem 3.3, Chapter II]{Yo}.)
On the other hand, we do not have any generalization of Euler's formulas (\ref{erf}),(\ref{ef}) nor cyclotomic units for general totally real fields.
Therefore the author believes that it is worthwhile to provide an alternative proof for Conjecture \ref{wsc} even in the known case $F=\mathbb Q$.
In fact, we obtain a weaker result (Corollaries \ref{crl1},\ref{crl2} in \S \ref{defofbeta}) 
\begin{equation} \label{application}
\begin{split}
&u_\mathbb Q(\sigma) \in \overline{\mathbb Q} \qquad (\sigma \in \mathrm{Gal}(\mathbb Q(\zeta_m)^+/\mathbb Q)), \\
&\tau (u_\mathbb Q(\sigma)) \equiv u_\mathbb Q(\tau \sigma) \bmod \mu_\infty \qquad (\sigma \in \mathrm{Gal}(\mathbb Q(\zeta_m)^+/\mathbb Q), 
\tau \in \mathrm{Gal}(\overline{\mathbb Q}/\mathbb Q))
\end{split}
\end{equation}
without Euler's formulas (\ref{erf}),(\ref{ef}) nor cyclotomic units. Here we denote by $\mu_\infty$ the group of roots of unity.

The outline of this paper is structured as follows.
We find that we can write a Stark unit over $\mathbb Q$ in terms of a product of special values of the beta function in \S \ref{gitmb}.
Therefore we can reduce the problem to an algebraicity property and a reciprocity law on such products.
In \S \ref{Roh}, we introduce a result of Rohrlich. It relates special values of the beta function to periods of Fermat curves.
Besides, we need some $p$-adic argument: 
We define a modified $p$-adic gamma function on $\mathbb Q_p$ in \S \ref{pg}.
Then we introduce a result of Coleman which expresses the absolute Frobenius action on the Fermat curves in \S \ref{col}.
In particular, we rewrite it in terms of our $p$-adic gamma function.
We note that when the  Fermat curve has a bad reduction, the expression in Theorem \ref{mt1}-\rm{(ii)} becomes simpler 
than Coleman's original formula \cite[Theorem 3.13]{Co}, although the root of unity ambiguity occurs.
In \S \ref{pp}, we define $p$-adic periods of Fermat curves and study their basic properties.
We will state the main results in \S \ref{defofbeta}: We define a ``period ring-valued beta function'' 
and provide a reciprocity law on its special values (Theorem \ref{main3}).
Moreover we derive (\ref{application}) from this reciprocity law.

\begin{rmk}
We will provide not only an alternative proof but also a refinement. 
By (\ref{suitocu}), Stark's conjecture in the case of $F=\mathbb Q$, $v = \infty$ implies the following reciprocity law on special values of the sine function:
\begin{equation} \label{rlofs}
\tau\left(\sin(\tfrac{a}{m}\pi)\right)=
\pm \sin\left(\tau(\tfrac{a}{m})\pi\right) \qquad (\tau \in \mathrm{Gal}(\overline{\mathbb Q}/\mathbb Q)).
\end{equation}
Here we define the action of $\mathrm{Gal}(\overline{\mathbb Q}/\mathbb Q)$ on $\mathbb Q \cap (0,1]$ by identifying 
the set of roots of unity $\mu_\infty$ and $\mathbb Q \cap (0,1]$, $\zeta_m^a \leftrightarrow \frac{a}{m}$.
Our main result (Theorem \ref{main3}) is a refinement of this reciprocity law (\ref{rlofs}), from the sine function to ``the period ring-valued beta function''. 
\end{rmk}

\begin{rmk}
The proof of our main results is based on some formulas by Rohrlich and Coleman.
Hiroyuki Yoshida formulated a conjecture \cite[Conjecture 3.9, Chapter III]{Yo} which is a generalization of Rohrlich's formula, 
from the rational number field to general totally real fields.
Yoshida and the author also conjectured its $p$-adic analogue in \cite{KY1},\cite{KY2}, which is a generalization of Coleman's formula.
These conjectures are one of the motivation of providing an alternative proof of Stark's conjecture in the case of $F=\mathbb Q$
by using Rohrlich's formula and Coleman's formula.
\end{rmk}

\section{The gamma function in terms of the beta function} \label{gitmb}

The beta function $B(\alpha,\beta)$ is defined by
\begin{equation*}
B(\alpha,\beta):=\int_0^1 t^{\alpha-1}(1-t)^{\beta-1} dt
\end{equation*}
for $0<\alpha,\beta \in \mathbb R$ and can be written in terms of the gamma function:
\begin{equation*}
B(\alpha,\beta)=\frac{\Gamma(\alpha)\Gamma(\beta)}{\Gamma(\alpha+\beta)}.
\end{equation*}
Conversely, we can write special values $\Gamma(\alpha)$ at rational numbers $\alpha \in \mathbb Q$ 
in terms of those of the beta function by virtue of the functional equation $\Gamma(z+1)=z\Gamma(z)$. 
For example, to obtain $\Gamma(\frac{1}{3})$, we compute the product
\begin{equation*}
B(\tfrac{1}{3},\tfrac{1}{3})^2B(\tfrac{2}{3},\tfrac{2}{3})
=\frac{\Gamma(\tfrac{1}{3})^2\Gamma(\tfrac{1}{3})^2}{\Gamma(\tfrac{2}{3})^2}\frac{\Gamma(\tfrac{2}{3})\Gamma(\tfrac{2}{3})}{\Gamma(\tfrac{4}{3})}
=3\Gamma(\tfrac{1}{3})^3.
\end{equation*}
For later use, we give an explicit formula of the product $\Gamma(\alpha)\Gamma(1-\alpha)$ for $\alpha \in \mathbb Q \cap (0,1)$ 
in terms of the beta function.

\begin{lmm} \label{sisprofb}
We write $\alpha \in \mathbb Q \cap (0,1)$ as $\alpha=\frac{a}{m}$ with $a,m \in \mathbb N$, $(a,m)=1$, $0<a<m$.
We take $t,m_0$ so that $m=2^tm_0$, $(2,m_0)=1$. 
When $\frac{a}{m_0} \neq 1$, we take the smallest $f_{m_0} \in \mathbb N$ satisfying $2^{f_{m_0}} \equiv 1 \bmod m_0$. Then we have 
\begin{equation*} 
\begin{split}
(\Gamma(\tfrac{a}{m})\Gamma(\tfrac{m-a}{m}))^{2^t(2^{f_{m_0}}-1)} 
=&\pm \prod_{k=1}^t\left(\tfrac{2^{k-1}m_0-a}{2^{k-1}m_0}
B(\tfrac{a}{2^km_0},\tfrac{a}{2^km_0})B(\tfrac{2^km_0-a}{2^km_0},\tfrac{2^km_0-a}{2^km_0})\right)^{2^{k-1}(2^{f_{m_0}}-1)} \\
&\times \prod_{l=0}^{f_{m_0}-1}\left(\tfrac{m_0-2^{l+1}a}{m_0}B(\tfrac{2^la}{m_0},\tfrac{2^la}{m_0})B(\tfrac{m_0-2^la}{m_0},\tfrac{m_0-2^la}{m_0})\right)^{2^{f_{m_0}-1-l}}.
\end{split}
\end{equation*}
When $\frac{a}{m_0} =1$ (i.e., $\alpha=\frac{a}{m}=\frac{1}{2^t}$), we have
\begin{equation*}
(\Gamma(\tfrac{1}{2^t})\Gamma(\tfrac{2^t-1}{2^t}))^{2^{t-1}}=B(\tfrac{1}{2},\tfrac{1}{2})\prod_{k=2}^t\left(\tfrac{2^{k-1}-1}{2^{k-1}}
B(\tfrac{1}{2^k},\tfrac{1}{2^k})B(\tfrac{2^k-1}{2^k},\tfrac{2^k-1}{2^k})\right)^{2^{k-2}}.
\end{equation*}
\end{lmm}

\begin{proof}
We put $\gamma(\alpha):=\Gamma(\alpha)\Gamma(1-\alpha)$, $\beta(\alpha):=(1-2\alpha)B(\alpha,\alpha)B(1-\alpha,1-\alpha)$.
Then we have $\beta(\alpha)=\frac{\gamma(\alpha)^2}{\gamma(2\alpha)}$ if $\alpha \neq \frac{1}{2}$.
First we assume that $2 \nmid m$. We take $f_m \in \mathbb N$ so that $2^{f_m} \equiv 1 \bmod m$. Then we obtain
\begin{equation*}
\prod_{l=0}^{f_m-1}\beta(\tfrac{2^la}{m})^{2^{f_m-1-l}} 
= \frac{\gamma(\frac{a}{m})^{2^{f_m}}}{\gamma(\frac{2a}{m})^{2^{f_m-1}}}\frac{\gamma(\frac{2a}{m})^{2^{f_m-1}}}{\gamma(\frac{4a}{m})^{2^{f_m-2}}}\dots
\frac{\gamma(\frac{2^{f_m-1}a}{m})^2}{\gamma(\frac{2^{f_m}a}{m})} =\pm \gamma\left(\tfrac{a}{m}\right)^{2^{f_m}-1},
\end{equation*}
since $\gamma(\tfrac{a}{m}) \bmod \{\pm 1\}$ depends only on $a \bmod m$ by $\Gamma(z+1)=z\Gamma(z)$.
Next, assume that $2 \nmid a$ and write $m=2^tm_0$ with $(2,m_0)=1$.
Then we have 
\begin{equation*}
\gamma(\tfrac{a}{m_0})\prod_{k=1}^t\beta(\tfrac{a}{2^km_0})^{2^{k-1}} 
=\gamma(\tfrac{a}{m_0})\frac{\gamma(\frac{a}{2m_0})^2}{\gamma(\frac{a}{m_0})}\frac{\gamma(\frac{a}{4m_0})^4}{\gamma(\frac{a}{2m_0})^2}\dots
\frac{\gamma(\frac{a}{2^tm_0})^{2^t}}{\gamma(\frac{a}{2^{t-1}m_0})^{2^{t-1}}}
=\gamma(\tfrac{a}{m})^{2^t}
\end{equation*}
if $\frac{a}{m_0}\neq 1$.
Combining these, we obtain the first formula.
The case of $\frac{a}{m_0}=1$ follows from a similar but simpler argument.
\end{proof}

\section{Periods of Fermat curves} \label{Roh}

Let $F_m$ be the $m$th Fermat curve defined by the affine equation $x^m+y^m=1$ ($3\leq m \in \mathbb N$). 
We consider differentials of the second kind $\eta_{\frac{i}{m},\frac{j}{m}}:=x^{i-1}y^{j-m}\mathit{d}x$ ($0<i,j<m$, $i+j\neq m$).
The following fact is well-known (\cite[Theorem in Appendix by Rohrlich]{Gr}). 
For all $\gamma \in H_1(F_m(\mathbb C),\mathbb Q)$, we have
\begin{equation} \label{pisb}
\frac{\int_{\gamma} \eta_{\frac{i}{m},\frac{j}{m}}}{B(\frac{i}{m},\frac{j}{m})} \in \mathbb Q(\zeta_m).
\end{equation}
Moreover we can take $\gamma_0 \in H_1(F_m(\mathbb C),\mathbb Q)$ ($\gamma_0=m\gamma_m$ with $\gamma_m$ in \cite[Proposition 4.9]{Ot}) so that
\begin{equation} \label{pisb2}
\int_{\gamma_0} \eta_{\frac{i}{m},\frac{j}{m}}=B(\tfrac{i}{m},\tfrac{j}{m}).
\end{equation}

\section{$p$-adic gamma functions} \label{pg}

We prepare $p$-adic analogues of the gamma function.
Morita \cite{Mo} constructed the $p$-adic gamma function $\Gamma_p \colon \mathbb Z_p \rightarrow \mathbb Z_p^\times$,
which is continuous and characterized by $\Gamma_p(n)=(-1)^n\prod_{k=1,\ (p,k)=1}^{n-1} k$ for $n \in \mathbb N$.
We note that the $p$-adic counterpart  of the formula (\ref{erf}) is ``degenerate'':
\begin{equation} \label{perf}
\Gamma_p(z)\Gamma_p(1-z)=\pm 1.
\end{equation}
For the proof, see \cite[Lemma 2.3]{GK}.

We define $\Gamma_p$ on $\mathbb Q_p-\mathbb Z_p$ as follows. For simplicity, assume that $p$ is odd.
We denote Iwasawa's $p$-adic $\log$ function by $\log_p$.
We define the $p$-adic exponential function $\exp_p(z)$ on $p\mathbb Z_p$ by the usual power series $\sum_{n=0}^\infty \frac{z^n}{n!}$.
In order to define $\exp_p(z)$ for $z \in \mathbb Q_p$, we choose values $\exp_p(\frac{1}{p^e}) \in \mathcal O_{\overline{\mathbb Q_p}}^\times
=\{z \in \overline{\mathbb Q_p} \mid |z|_p=1\}$ for $0\leq e \in \mathbb Z$ so that 
\begin{equation} \label{rofep}
\left(\exp_p(\tfrac{1}{p^{e+1}})\right)^p=\exp_p(\tfrac{1}{p^e}).
\end{equation}
For any $z \in \mathbb Q_p$, we can write $z=\tfrac{n}{p^e}+z_0$ with $n,e \in \mathbb Z$, $0\leq e$, $0\leq n < p^{e+1}$, $z_0 \in p\mathbb Z_p$.
Then we define
\begin{equation*}
\exp_p(\tfrac{n}{p^e}+z_0):=\exp_p(\tfrac{1}{p^e})^n \exp_p(z_0).
\end{equation*}
We write $z'=\frac{n'}{p^{e'}}+z_0'$ similarly for $z' \in \mathbb Q_p$ with $|z-z'|_p\leq |p|_p$. Then we have $\frac{n}{p^e}=\frac{n'}{p^{e'}}$.
Therefore $\exp_p(z)$ is well-defined and continuous on $z \in \mathbb Q_p$.
Moreover it satisfies
\begin{equation} \label{propofep}
\begin{split}
\exp_p(z_1+z_2)&=\exp_p(z_1)\exp_p(z_2) \qquad (z_1,z_2 \in \mathbb Q_p), \\
\exp_p(\log_p(z))&=z^* \qquad (z \in \mathbb Q_p^\times).
\end{split}
\end{equation}
Here $z^*=z\omega(zp^{-\mathrm{ord}_p\,z})^{-1}p^{-\mathrm{ord}_p\,z}$ is defined by the usual decomposition:
\begin{equation*}
\begin{array}{ccccccc}
\mathbb Q_p^\times & = & \mu_{p-1} & \times & p^{\mathbb Z} & \times & (1+p\mathbb Z_p), \\
z & = & \omega(zp^{-\mathrm{ord}_pz}) & \times & p^{\mathrm{ord}_p\,z} & \times & z^* 
\end{array}
\end{equation*}
with $\mu_{p-1}$ the set of $(p-1)$st roots of unity, $\omega$ the Teichm\"uller character.
We shall give a proof of (\ref{propofep}): 
Write $z_i=\frac{n_i}{p^{e_i}}+z_{i,0}$ ($i=1,2$) in a manner similar to the above.
We may assume that $e_1\leq e_2$.
Then we have
\begin{equation*}
\exp_p(z_1+z_2)=
\begin{cases}
\left(\exp_p(\tfrac{1}{p^{e_2}})\right)^{n_1p^{e_2-e_1}+n_2} \exp_p(z_{1,0}+z_{2,0}) & (\frac{n_1}{p^{e_1}}+\frac{n_2}{p^{e_2}}<p) \\
\left(\exp_p(\tfrac{1}{p^{e_2}})\right)^{n_1p^{e_2-e_1}+n_2-p^{e_2+1}} \exp_p(p+z_{1,0}+z_{2,0}) & (\frac{n_1}{p^{e_1}}+\frac{n_2}{p^{e_2}}\geq p)
\end{cases}
\end{equation*}
by definition.
Since $\exp_p$ on $p \mathbb Z_p$ is a group homomorphism by a property of the power series $\sum_{n=0}^\infty \frac{z^n}{n!}$, 
the first equality of (\ref{propofep}) follows from (\ref{rofep}).
The latter equality follows from the fact that $\log_pz=\log_pz^*$ and $\exp_p\circ \log_p$ is the identity map on $1+p\mathbb Z_p$.

\begin{rmk}
``The root of unity ambiguity'' occurs when we choose $\exp_p(\frac{1}{p^e})$. 
\end{rmk}

\begin{lmm} \label{pgamma}
There exists a continuous function $\Gamma_p\colon \mathbb Q_p-\mathbb Z_p \rightarrow \mathcal O_{\overline{\mathbb Q_p}}^\times$
satisfying
\begin{equation} \label{feqofpg}
\Gamma_p(z+1)= z^* \Gamma_p(z), \qquad
\Gamma_p(2z)=2^{2z-\frac{1}{2}} \Gamma_p(z)\Gamma_p(z+\tfrac{1}{2}) \qquad (z \in \mathbb Q_p-\mathbb Z_p).
\end{equation}
Here we put $2^{2z-\frac{1}{2}}:=\exp_p((2z-\frac{1}{2})\log_p2)$.
Moreover, such a function $\Gamma_p $ is unique up to multiplication by a root of unity.
Strictly speaking, if continuous functions $\Gamma_p ,\Gamma_p'$ satisfy (\ref{feqofpg}), 
then for all $z \in \mathbb Q_p-\mathbb Z_p$ we have $\Gamma_p(z)\equiv \Gamma_p'(z) \bmod \mu_\infty$.
\end{lmm}

\begin{proof}
In \cite{Ka1}, the $p$-adic logarithmic gamma function $L\Gamma_{p,1}(a,(a_1))$ ($a\in \mathbb Q_p$, $a_1 \in \mathbb Q_p^\times$) 
is defined as the derivative value at $s=0$ of the $p$-adic Hurwitz zeta function $\zeta_{p,1}(s,(a_1),a)$ (\cite[(5.10)]{Ka1}).
We omit the precise definition but we write $L\Gamma_{p,1}(a,(a_1))$ in the form of \cite[(5.12)]{Ka1}.
We introduce some notations and definitions:
For a function $f(X)\colon \mathbb Z_p \rightarrow \mathbb Q_p$, we define
\begin{equation*}
J_X(f(X)):=\lim_{l\rightarrow \infty} \frac{1}{p^l}\sum_{n=0}^{p^l-1}f(n)
\end{equation*}
if it converges in $\mathbb Q_p$. We easily see that 
\begin{equation*}
J_X((aX+b)^n)=a^nB_n(\tfrac{b}{a}) \qquad (0\leq n \in \mathbb Z, a \in \mathbb Q_p^\times, b \in \mathbb Q_p),
\end{equation*}
where $B_n(x)$ it the $n$th Bernoulli polynomial (\cite[Lemma 5.1]{Ka1}). 
For $a \in \mathbb Z_p^\times$, $0\neq a_1 \in p\mathbb Z_p$, we define the function $f_{a,a_1}(X) \colon \mathbb Z_p \rightarrow \mathbb Q_p$ by 
\begin{equation*}
\begin{split}
f_{a,a_1}(X)&:=\frac{-(a_1X+a)}{a_1}(1-\log_p(a_1X+a)) \\ 
&=\frac{-1}{a_1}\left((a_1X+a)+\sum_{k=1}^\infty \frac{(-1)^k}{k}\left(\sum_{l=0}^k\frac{k!}{l!(k-l)!}\omega(a)^{-l}(a_1X+a)^{l+1}(-1)^{k-l}\right)\right).
\end{split}
\end{equation*}
We put
\begin{equation} \label{lgp}
\begin{split}
L\Gamma_{p,1}(a,(a_1))&:=J_X(f_{a,a_1}(X)) \\
&=-B_1(\tfrac{a}{a_1})-\sum_{k=1}^\infty \frac{(-1)^k}{k}\left(\sum_{l=0}^k\frac{k!}{l!(k-l)!}\omega(a)^{-l}a_1^lB_{l+1}(\tfrac{a}{a_1})(-1)^{k-l}\right).
\end{split}
\end{equation}
We see that for $c \in \mathbb Z_p^\times$
\begin{equation} \label{fflgp2}
L\Gamma_{p,1}(ca,(ca_1))=L\Gamma_{p,1}(a,(a_1))+B_1(\tfrac{a}{a_1})\log_pc
\end{equation}
since $f_{ca,ca_1}(X)=\frac{(a_1X+a)}{a_1}\log_pc+f_{a,a_1}(X)$.
The equation \cite[(5.12)]{Ka1} (in the case of $r=1$) is stated only when $|a-1|_p<1$, 
but it is also valid for all $a \in \mathbb Z_p^\times$ since we have $L\Gamma_{p,1}(a,(a_1))=L\Gamma_{p,1}(\omega(a)^{-1}a,(\omega(a)^{-1}a_1))$ 
by (\ref{fflgp2}).
So the definition (\ref{lgp}) of $L\Gamma_{p,1}$ in this paper is consistent with that in \cite{Ka1}.
The function $L\Gamma_{p,1}(a,(a_1))$ is continuous on $a \in \mathbb Z_p^\times$, $a_1 \in p\mathbb Z_p$ by \cite[Lemma 5.4-1]{Ka1}.

Now we define the $p$-adic $\Gamma$ function on $\mathbb Q_p-\mathbb Z_p$. For $z \in \mathbb Z_p^\times$, $e \in \mathbb N$, we put
\begin{equation} \label{defofgp}
\Gamma_p(\tfrac{z}{p^e}):=\exp_p(L\Gamma_{p,1}(z,(p^e))).
\end{equation}
It is continuous since so are $\exp_p$, $L\Gamma_{p,1}$.
One can derive the desired properties from the expression (\ref{lgp}) and some properties of Bernoulli polynomials.
Indeed, the well-known formulas
\begin{equation*}
B_n(x+1)=B_n(x)+nx^{n-1}, \qquad
B_n(mx)=m^{n-1}\sum_{k=0}^{m-1}B_n(x+\tfrac{k}{m})\quad (m \in \mathbb N)
\end{equation*}
imply that for $a \in \mathbb Z_p^\times$, $a_1 \in p\mathbb Z_p$, we have
\begin{equation} \label{fflgp1}
\begin{split}
&L\Gamma_{p,1}(a+a_1,(a_1))-L\Gamma_{p,1}(a,(a_1)) \\
&=-1-\sum_{k=1}^\infty \frac{(-1)^k}{k}\left(\sum_{l=0}^k\frac{k!}{l!(k-l)!} (l+1)(\omega(a)^{-1}a)^l(-1)^{k-l}\right) \\
&=-1-\sum_{k=1}^\infty \frac{(-1)^k}{k}(\omega(a)^{-1}a-1)^k-(-\omega(a)^{-1}a)\sum_{k=1}^\infty (1-\omega(a)^{-1}a)^{k-1}=\log_p(a), \\
&L\Gamma_{p,1}(ma,(a_1))=\sum_{k=0}^{m-1}L\Gamma_{p,1}(ma+ka_1,(ma_1))\quad (m\in \mathbb N,\ (p,m)=1).
\end{split}
\end{equation}
By (\ref{fflgp2}) and (\ref{fflgp1}), we have for $z \in \mathbb Z_p^\times$, $e,m\in \mathbb N$ with $(p,m)=1$
\begin{eqnarray}
L\Gamma_{p,1}(z+p^e,(p^e))&=&\log_pz+L\Gamma_{p,1}(z,(p^e)), \label{prpoflgp1} \\
L\Gamma_{p,1}(mz,(p^e))&=&\sum_{k=0}^{m-1}\left((\tfrac{z}{p^e}+\tfrac{k}{m}-\tfrac{1}{2})\log_pm+L\Gamma_{p,1}(z+\tfrac{kp^e}{m},(p^e))\right). \label{prpoflgp2} 
\end{eqnarray}
Then the assertions (\ref{feqofpg}) follow from (\ref{propofep}),(\ref{defofgp}),(\ref{prpoflgp1}) and (\ref{prpoflgp2}) in case of $m=2$.
In order to prove the uniqueness of $\Gamma_p$ up to $\mu_\infty$, it suffices to show that $\log_p$ of $\Gamma_p$ is unique,
since the kernel of $\log_p$ is generated by $\mu_\infty$ and rational powers of $p$.
(Note that we assumed that $\Gamma_p$ takes values in $\mathcal O_{\overline{\mathbb Q_p}}^\times$.)
To do this, let $f_1,f_2$ be continuous functions on $\mathbb Q_p-\mathbb Z_p$ satisfying 
$f_i(z+1)=\log_pz+f_i(z)$, $f_i(2z)=\log_p(2^{2z-\frac{1}{2}})+f_i(z)+f_i(z+\frac{1}{2})$.
Then $h(z):=f_1(z)-f_2(z)$ satisfies $h(z+1)=h(z)$, $h(2z)=h(z)+h(z+\frac{1}{2})$.
Since $h$ is continuous, $h(z+1)=h(z)$ means that $h(z+a)=h(z)$ for all $a \in \mathbb Z_p$.
Hence $h(2z)=h(z)+h(z+\frac{1}{2})=2h(z)$.
Write $z=\frac{z_0}{p^e}$ with $z_0 \in \mathbb Z_p^\times$, $e\in \mathbb N$.
Since $(2,p)=1$, there exists $f,m \in \mathbb N$ satisfying $2^f=1+mp^e$. Then we get 
$2^fh(z)=h(2^fz)=h(z+mp^ez)=h(z+mz_0)=h(z)$.
Therefore we obtain $f_1(z)-f_2(z)=h(z)=0$ as desired.
\end{proof}

\begin{dfn} \label{gp}
We fix $\Gamma_p$ on $\mathbb Q_p-\mathbb Z_p$ satisfying (\ref{feqofpg})
and define the $p$-adic gamma function $\Gamma_p$ on $\mathbb Q_p$ by gathering our $\Gamma_p$ on $\mathbb Q_p-\mathbb Z_p$ 
and Morita's $\Gamma_p$ on $\mathbb Z_p$.
We define two kinds of $p$-adic beta functions 
$B(*,*),B\langle*,*\rangle \colon \mathbb Q \times \mathbb Q \rightarrow \mathcal O_{\overline{\mathbb Q_p}}^\times$ by 
\begin{equation*}
B_p(\alpha,\beta):=\frac{\Gamma_p(\alpha)\Gamma_p(\beta)}{\Gamma_p(\alpha +\beta)}, \quad
B_p\langle\alpha,\beta\rangle:=\frac{\Gamma_p(\langle\alpha\rangle)\Gamma_p(\langle\beta\rangle)}{\Gamma_p(\langle \alpha +\beta\rangle)}.
\end{equation*}
Here we denote the fractional part $\in (0,1]$ of $\alpha \in \mathbb Q$ by $\langle \alpha \rangle$.
\end{dfn}

\begin{rmk}
The functional equations (\ref{feqofpg}) are $p$-adic analogues of classical formulas
\begin{equation} \label{feqofg}
\Gamma(z+1)= z\Gamma(z), \qquad
\Gamma(2z)= \frac{2^{2z-\frac{1}{2}}}{\sqrt{2\pi}}\Gamma(z)\Gamma(z+\tfrac{1}{2}).
\end{equation}
We note that Artin \cite{Ar1,Ar2} showed that the classical gamma function is also characterized by functional equations (\ref{feqofg}) and some conditions.
\end{rmk}

\begin{rmk}
By (\ref{prpoflgp2}), we also obtain the multiplication formula for $m \in \mathbb N$ with $(p,m)=1$. Namely, for $z \in \mathbb Q_p-\mathbb Z_p$ we have
\begin{equation*}
\Gamma_p(mz)= \prod_{k=0}^{m-1}m^{z+\frac{k}{m}-\frac{1}{2}}\Gamma_p(z+\tfrac{k}{m})
\end{equation*}
with $m^{z+\frac{k}{m}-\frac{1}{2}}:=\exp_p((z+\frac{k}{m}-\frac{1}{2})\log_pm)$.
We need some adjustments for $m$ with $p \mid m$.
In particular, one can show that for $z \in \mathbb Q_p$
\begin{equation*}
\begin{split}
\Gamma_p(pz)&=\prod_{k=0}^{p-1}\Gamma_p(z+\tfrac{k}{p}) \qquad (pz \notin \mathbb Z_p), \\
\Gamma_p(pz)&\equiv \prod_{k=0,\ pz+k\in \mathbb Z_p^\times}^{p-1}\Gamma_p(z+\tfrac{k}{p}) \bmod \mu_\infty \qquad (pz \in \mathbb Z_p).
\end{split}
\end{equation*}
Note that the latter formula gives a relation between Morita's $\Gamma_p$ on $\mathbb Z_p$ and our $\Gamma_p$ on $\mathbb Q_p-\mathbb Z_p$.
We provide a brief sketch of the proof: When $pz \notin \mathbb Z_p$, we can write $z=\frac{z_0}{p^e}$ with $z_0 \in \mathbb Z_p^\times,e \geq 2$.
Then $\Gamma_p(pz)=\exp_p(L\Gamma_{p,1}(z_0,(p^{e-1})))$, $\Gamma_p(z+\frac{k}{p})=\exp_p(L\Gamma_{p,1}(z_0+kp^{e-1},(p^e)))$. 
Therefore we only need the formula $L\Gamma_{p,1}(z_0,(p^{e-1}))=\sum_{k=0}^{p-1}L\Gamma_{p,1}(z_0+kp^{e-1},(p^e))$.
It follows form (\ref{lgp}) similarly to the proof of (\ref{prpoflgp2}).
When $pz \in \mathbb Z_p$, it follows from \cite[Lemma 5.5]{Ka1} since 
$L\Gamma_{p,1}(a,(1))$ ($a \in \mathbb Z_p$) in \cite{Ka1} was defined as $\sum_{k=0,\ pz+k\in \mathbb Z_p^\times}^{p-1}L\Gamma_{p,1}(a+k,(p))$.
\end{rmk}

Coleman \cite{Co} also generalized Morita's $\Gamma_p$ to a continuous function on $\mathbb Q_p$.
We denote it by $\Gamma_{\mathrm{col}}$ in this paper.
It is characterized by the following conditions (\cite[the sentence containing the equation (2.2)]{Co}):
\begin{equation*}
\begin{split}
\Gamma_{\mathrm{col}}(z+1)&=z^* \Gamma_{\mathrm{col}}(z) \qquad (z \in \mathbb Q_p -\mathbb Z_p), \\
\Gamma_{\mathrm{col}}(z)&=1 \qquad (z \in \mathbb Z[\tfrac{1}{p}]\cap [0,1)).
\end{split}
\end{equation*}
For later use, we compare Coleman's $\Gamma_{\mathrm{col}}$ and our $\Gamma_p$.
For $z \in \mathbb Q_p-\mathbb Z_p$, we write $z=z_p+z_0$ with $z_p \in \mathbb Z[\frac{1}{p}] \cap (0,1)$, $z_0 \in \mathbb Z_p$.
Then we have for each $\sharp=p,\mathrm{col}$
\begin{equation*}
\Gamma_\sharp(z)=\Gamma_\sharp(z_p+z_0)=\lim_{\mathbb N \ni n \rightarrow z_0}\Gamma_\sharp(z_p+n)
= \Gamma_\sharp(z_p)\lim_{\mathbb N \ni n \rightarrow z_0}\prod_{l=0}^{n-1}(z_p+l)^*.
\end{equation*}
Hence by $\Gamma_{\mathrm{col}}(z_p)=1$, we have
\begin{equation} \label{pcol}
\Gamma_{p}(z)= \Gamma_{\mathrm{col}}(z)\Gamma_{p}(z_p) \qquad (z \in \mathbb Q_p -\mathbb Z_p).
\end{equation}

\section{The Frobenius matrices of Fermat curves} \label{col}

In this section, we recall Coleman's results in \cite{Co}.
In \cite{Co}, the semi-linear action $\rho_{\mathrm{cris}}$ of the (crystalline) Weil group on the de Rham cohomology of a curve with arboreal reduction 
is introduced.
We express this action on Fermat curves in terms of the comparison of some cohomologies for later use.
Let $\mathbb Q_p^{\mathrm{ur}}$ be the maximal unramified extension of $\mathbb Q_p$ and 
$W_p \subset \mathrm{Gal}(\overline{\mathbb Q_p}/\mathbb Q_p)$ the Weil group.
Then for all $\tau \in W_p$, there exists $n \in \mathbb Z$ such that $\tau|_{\mathbb Q_p^{\mathrm{ur}}}$ is the $n$th power of the Frobenius automorphism.
We put $\deg \tau:=n$. 

Let $J_m$ be the Jacobian variety of the $m$th Fermat curve $F_m$.
Even if $F_m$ dose not have good reduction at $p$ (i.e., the case of $p \mid m$), $J_m$ has potentially good reduction since it has CM.
Hence there exist a finite extension field $K$ of $\mathbb Q_p$ and a smooth model $\mathcal J_m$ over $\mathcal O_K$ with the canonical isomorphism
\begin{equation} \label{pdr}
\mathrm{H_{dR}^1}(F_m,\mathbb Q)\otimes_\mathbb Q K \cong \mathrm{H_{cris}^1}(\mathcal J_m\times_{\mathcal O_K} \mathbb F_K,W_K)\otimes_{W_K} K.
\end{equation}
Here we denote by $\mathcal O_K,\mathbb F_K$ and $W_K$ the ring of integers, its residue field and the Witt ring over $\mathbb F_K$ respectively.
We may assume $K$ is normal over $\mathbb Q_p$.
We denote by $\Phi_\textrm{cris}$ the action of the absolute Frobenius on $\mathrm{H_{cris}^1}(\mathcal J_m\times_{\mathcal O_K} \mathbb F_K,W_K)$.
Then for $\tau \in W_p$ we may consider the action $\Phi_\textrm{cris}^{\deg \tau} \otimes \tau$ on the right hand side of (\ref{pdr}).
We denote by $\Phi_\tau$ the corresponding action on $\mathrm{H_{dR}^1}(F_m,\mathbb Q)\otimes_\mathbb Q K$.
Our $\Phi_\tau$ is equal to $\rho_{\mathrm{cris}}(\tau)$ in \cite{Co}.
Therefore it is also equal to $\Phi^*(\tau)$ in \cite{Co} on $\mathrm{H_{dR}^1}(F_m,\mathbb Q)$.
We need a few more notations in \cite{Co}:
Let $i,j,m \in \mathbb N$ with $0<i,j<m$, $i+j\neq m$.
We put $\varepsilon(\frac{i}{m},\frac{j}{m}):=\langle\frac{i}{m}\rangle+\langle\frac{j}{m}\rangle-\langle\frac{i}{m}+\frac{j}{m}\rangle$.
That is
\begin{equation*}
\langle\tfrac{i}{m}+\tfrac{j}{m}\rangle=
\begin{cases}
\frac{i}{m}+\frac{j}{m} & (i+j<m) \\
\frac{i}{m}+\frac{j}{m}-1 &  (i+j>m)
\end{cases}, \qquad
\varepsilon(\tfrac{i}{m},\tfrac{j}{m})=
\begin{cases}
0 & (i+j<m) \\
1 & (i+j>m)
\end{cases}.
\end{equation*}
The element $v_{m,r,s} \in \mathrm{H_{dR}^1}(F_m,\mathbb Q)$ in \cite{Co} is the class of 
\begin{equation} \label{differ}
m\langle\tfrac{i}{m}+\tfrac{j}{m}\rangle^{\varepsilon(\frac{i}{m},\frac{j}{m})}\eta_{\frac{i}{m},\frac{j}{m}}=
\begin{cases}
m\eta_{\frac{i}{m},\frac{j}{m}} & (i+j<m) \\
(i+j-m)\eta_{\frac{i}{m},\frac{j}{m}} & (i+j>m)
\end{cases}
\end{equation}
for $i,j$ with $r=\frac{i}{m}$, $s=\frac{j}{m}$.
Coleman calculated $\Phi^*(\tau)(v_{m,r,s})$ explicitly in \cite{Co} as follows.
We recall that we define the action of $\mathrm{Gal}(\overline{\mathbb Q_p}/\mathbb Q_p)$
on $\mathbb Q \cap (0,1]$ by identifying $\mathbb Q \cap (0,1] = \mu_\infty$, $\frac{a}{m} \leftrightarrow \zeta_m^a$.

\begin{thm}[{\cite[Theorems 1.7, 3.13]{Co}}] \label{mt1}
By abuse of notation, we write $\eta_{\frac{i}{m},\frac{j}{m}}$ for its class in $\mathrm H^1_{\mathrm{dR}}(F_m,\mathbb Q)$ ($0<i,j<m$, $i+j\neq m$).
Then for $\tau \in W_p$ we can write
\begin{equation*}
\Phi_\tau (\eta_{\frac{i}{m},\frac{j}{m}})=\gamma(\tau,\tfrac{i}{m},\tfrac{j}{m})\eta_{\tau(\frac{i}{m}),\tau(\frac{j}{m})}
\end{equation*}
with $\gamma(\tau,\frac{i}{m},\frac{j}{m}) \in \mathbb Q_p$. 
Moreover the values $\gamma(\tau,\frac{i}{m},\frac{j}{m})$ satisfy the following conditions.
\begin{enumerate}
\item Assume that $(p,m)=1$. Then we have for $\tau \in W_p$ with $\deg \tau=1$ 
\begin{equation*}
\gamma(\tau,\tfrac{i}{m},\tfrac{j}{m})=\frac{p^{1-\varepsilon(\frac{i}{m},\frac{j}{m})}}{\langle\frac{i}{m}+\frac{j}{m}\rangle^{\varepsilon(\frac{i}{m},\frac{j}{m})}}
\frac{(-1)^{\varepsilon(\tau(\frac{i}{m}),\tau(\frac{j}{m}))}\langle\tau(\frac{i}{m})+\tau(\frac{j}{m})\rangle^{\varepsilon(\tau(\frac{i}{m}),\tau(\frac{j}{m}))}}
{B_p\langle \tau(\frac{i}{m}),\tau(\frac{j}{m})\rangle}.
\end{equation*}
\item Assume that $p>2$, $p \mid m$, $(p,ij(i+j))=1$.
Then we have for $\tau \in W_p$ with $\deg \tau=1$ 
\begin{equation*}
\gamma(\tau,\tfrac{i}{m},\tfrac{j}{m}) \equiv p^{\frac{1}{2}}\frac{B_p(\frac{i}{m},\frac{j}{m})}{B_p(\tau(\frac{i}{m}),\tau(\frac{j}{m}))} \bmod \mu_{\infty}.
\end{equation*}
\end{enumerate}
\end{thm}

\begin{proof}
Our $\gamma(\tau,\frac{i}{m},\frac{j}{m})$ is Coleman's 
$\frac{\langle\tau(\frac{i}{m})+\tau(\frac{j}{m})\rangle^{\varepsilon(\tau(\frac{i}{m}),\tau(\frac{j}{m}))}}
{\langle\frac{i}{m}+\frac{j}{m}\rangle^{\varepsilon(\frac{i}{m},\frac{j}{m})}}\beta_\tau(\tau(\frac{i}{m}),\tau(\frac{j}{m}))$ by (\ref{differ}), \cite[Proposition 1.4]{Co}.
Therefore the former assertion follows from \cite[Theorem 1.7]{Co} immediately.
In the latter case, \cite[Proposition 3.12, Theorem 3.13]{Co} state
\begin{equation} \label{gammabeta}
\begin{split}
&\gamma(\tau,\tau^{-1}(\tfrac{i}{m}),\tau^{-1}(\tfrac{j}{m})) \\ 
&\equiv \frac{\langle\frac{i}{m}+\frac{j}{m}\rangle^{\varepsilon(\frac{i}{m},\frac{j}{m})}}
{\langle\tau^{-1}(\frac{i}{m})+\tau^{-1}(\frac{j}{m})\rangle^{\varepsilon(\tau^{-1}(\frac{i}{m}),\tau^{-1}(\frac{j}{m}))}}
\frac{D_{\tau}(\frac{i}{m})D_{\tau}(\frac{j}{m})\Gamma_{\tau}(\frac{i}{m})\Gamma_{\tau}(\frac{j}{m})}{D_{\tau}(\langle\frac{i}{m}+\frac{j}{m}\rangle)
\Gamma_{\tau}(\langle\frac{i}{m}+\frac{j}{m}\rangle)}
\bmod \mu_{\infty}.
\end{split}
\end{equation}
The notations are as follows: Coleman defined functions $D_{\tau},A_{\tau},\Gamma_{\tau}$ satisfying
\begin{equation*}
\begin{split}
D_{\tau}(z)&\equiv \prod_{k=1}^f(A_\tau((z/2^k)_p)^{\frac{2^{k-1}}{2^f-1}} \bmod \mu_{\infty}, \\
A_{\tau}(z)&=(2^{-\tau^{-1}((2z)_p)\tau+(2z)_p})^*
\frac{\Gamma_{\mathrm{col}}(z_p+\frac{(-1)^{(2z)_p}}{2})}{\Gamma_{\mathrm{col}}(\tau^{-1}(z_p)+\frac{(-1)^{\tau^{-1}((2z)_p)}}{2})}, \\
\Gamma_{\tau}(z)&=-p\Gamma_p(\tfrac{1}{2})(-p)^\frac{\tau-p}{2(p-1)}(2z)^\frac{\tau-1}{2} \kappa(z)^{\tau^{-1}(z)\tau-z}
\frac{\Gamma_{\mathrm{col}}(\tau^{-1}(z))}{\Gamma_{\mathrm{col}}(z)}
\end{split}
\end{equation*}
for $z \in \mathbb Q^\times \cap (0,1)$ with $\mathrm{ord}_p\,z <0$ (\cite[p179, p182]{Co}).
Here $f$ is the order of $2$ in $(\mathbb Z_p/ z^{-1}\mathbb Z_p)^\times$, 
$\kappa(z):=\frac{z}{z^*}$ and $z_p$ is the unique element in $\mathbb Z[\frac{1}{p}] \cap (0,1]$ satisfying $z\equiv z_p \bmod \mathbb Z_p$.
We define $x^{\frac{r}{m}}$ with $x \in \mathbb Q_p, r \in \mathbb Z[W_p], m \in \mathbb N$ as $(x^{\frac{1}{m}})^r$ by taking 
an $m$th root $x^{\frac{1}{m}} \in \overline{\mathbb Q_p}$ of $x$.
In order to define $(-1)^{z_p}$, we regard $z_p$ as its image 
under the composite map $\mathbb Z[\frac{1}{p}]\subset \mathbb Z_2 \rightarrow \mathbb Z_2/2\mathbb Z_2=\{0,1\}$.
We put
\begin{equation*}
\begin{split}
D(z)&:=\prod_{k=1}^fA((z/2^k)_p)^{\frac{2^{k-1}}{2^f-1}} \bmod \mu_{\infty}, \\
A(z)&:=2^{(2z)_p} \Gamma_{\mathrm{col}}(z_p+\tfrac{(-1)^{(2z)_p}}{2}) \bmod \mu_{\infty}.
\end{split}
\end{equation*}
Then by (\ref{pcol}) we have 
\begin{equation*}
A(z) \equiv 2^{(2z)_p} \frac{\Gamma_p(z_p+\frac{(-1)^{(2z)_p}}{2})}{\Gamma_p(z_p)} \bmod \mu_{\infty}.
\end{equation*}
Since $\Gamma_p(2(z_p))= 2^{2(z_p)-\frac{1}{2}} \Gamma_p(z_p)\Gamma_p(z_p+\frac{1}{2})$ by (\ref{feqofpg}), we get
\begin{equation*}
A(z) 
\equiv 2^{\frac{1}{2}+(2z)_p-2(z_p)} \frac{\Gamma_p(z_p+\frac{(-1)^{(2z)_p}}{2})\Gamma_p(2(z_p))}{\Gamma_p(z_p)^2\Gamma_p(z_p+\frac{1}{2})} \bmod \mu_{\infty}.
\end{equation*}
If $2(z_p)=(2z)_p$ then $(-1)^{(2z)_p}=1$ so we have
\begin{equation} \label{azgp}
A(z) \equiv 2^{\frac{1}{2}} \frac{\Gamma_p((2z)_p)}{\Gamma_p(z_p)^2} \bmod \mu_{\infty}.
\end{equation}
If $2(z_p)=(2z)_p+1$ then $(-1)^{(2z)_p}=-1$ so we get
\begin{equation*}
A(z) \equiv 2^{\frac{1}{2}-1} \frac{\Gamma_p(z_p-\frac{1}{2})\Gamma_p((2z)_p+1)}{\Gamma_p(z_p)^2\Gamma_p(z_p+\frac{1}{2})} \bmod \mu_{\infty}.
\end{equation*}
Hence the formula (\ref{azgp}) is also valid in this case 
since $\Gamma_p((2z)_p+1) = ((2z)_p)^*\Gamma_p((2z)_p)=(2(z_p)-1)^*\Gamma_p((2z)_p)$, 
$\Gamma_p(z_p+\frac{1}{2})= (z_p-\frac{1}{2})^*\Gamma_p(z_p-\frac{1}{2})$.
By (\ref{pcol}),(\ref{azgp}) and $(z/2^f)_p=z_p$, we see that
\begin{equation*}
\begin{split}
D(z)&\equiv 
\left (2^{\frac{1}{2}} \frac{\Gamma_p(z_p)}{\Gamma_p((z/2)_p)^2}\right )^{\frac{1}{2^f-1}}
\left (2^{\frac{1}{2}} \frac{\Gamma_p((z/2)_p)}{\Gamma_p((z/2^2)_p)^2}\right )^{\frac{2}{2^f-1}}
\dots
\left (2^{\frac{1}{2}} \frac{\Gamma_p((z/2^{f-1})_p)}{\Gamma_p((z/2^f)_p)^2}\right )^{\frac{2^{f-1}}{2^f-1}} \\
& \equiv \frac{2^{\frac{1}{2}}}{\Gamma_p(z_p)} \equiv \frac{2^{\frac{1}{2}}\Gamma_{\mathrm{col}}(z)}{\Gamma_p(z)} \bmod \mu_{\infty}.
\end{split}
\end{equation*}
We easily see that $D_\tau(z)\equiv \frac{D(z)}{D(\tau^{-1}(z))}$, $\Gamma_p(\frac{1}{2}) \equiv 1$ (by (\ref{perf})), 
$(-p)^\frac{\tau-p}{2(p-1)}\equiv p^{\frac{-1}{2}}$, 
$(2z)^\frac{\tau-1}{2} \equiv 1$, $\kappa(z)^{\tau^{-1}(z)\tau-z} \equiv \frac{p^{\tau^{-1}(z)\mathrm{ord}_p\,z}}{p^{z\mathrm{ord}_p\,z}} \bmod \mu_\infty$.
Therefore we can write
\begin{equation*}
D_\tau(z)\Gamma_\tau(z) \equiv p^{\frac{1}{2}}
\frac{p^{\tau^{-1}(z)\mathrm{ord}_p\,z}}{p^{z\mathrm{ord}_p\,z}}
\frac{\Gamma_p(\tau^{-1}(z))}{\Gamma_p(z)} \bmod \mu_\infty.
\end{equation*}
By substituting this into (\ref{gammabeta}), we get the desired result
\begin{equation*}
\gamma(\tau,\tau^{-1}(\tfrac{i}{m}),\tau^{-1}(\tfrac{j}{m})) 
\equiv p^{\frac{1}{2}} 
\frac{\Gamma_p(\tau^{-1}(\frac{i}{m}))}{\Gamma_p(\frac{i}{m})}
\frac{\Gamma_p(\tau^{-1}(\frac{j}{m}))}{\Gamma_p(\frac{j}{m})}
\frac{\Gamma_p(\frac{i}{m}+\frac{j}{m})}{\Gamma_p(\tau^{-1}(\frac{i}{m})+\tau^{-1}(\frac{j}{m}))} \bmod \mu_{\infty}
\end{equation*}
since $\mathrm{ord}_p\,\frac{i}{m}=\mathrm{ord}_p\,\frac{j}{m}=\mathrm{ord}_p\,\langle\frac{i}{m}+\frac{j}{m}\rangle$, 
$\Gamma_p(\frac{i}{m}+\frac{j}{m})=(\langle\frac{i}{m}+\frac{j}{m}\rangle^*)^{\varepsilon(\frac{i}{m},\frac{j}{m})}\Gamma_p(\langle\frac{i}{m}+\frac{j}{m}\rangle)
\equiv \frac{\langle\frac{i}{m}+\frac{j}{m}\rangle^{\varepsilon(\frac{i}{m},\frac{j}{m})}}
{p^{\varepsilon(\frac{i}{m},\frac{j}{m})\mathrm{ord}_p\,\langle\frac{i}{m}+\frac{j}{m}\rangle}}\Gamma_p(\langle\frac{i}{m}+\frac{j}{m}\rangle)$.
\end{proof}

\section{$p$-adic periods of Fermat curves} \label{pp}

We rewrite Theorem \ref{mt1} by using $p$-adic periods of Fermat curves.
Let $B_{\mathrm{cris}},B_{\mathrm{dR}}$ be the $p$-adic period rings introduced by Fontaine.
We may consider $\overline{\mathbb Q_p},B_{\mathrm{cris}}$ are subrings of $B_{\mathrm{dR}}$.
Then for any subfield $K$ of $\overline{\mathbb Q_p}$, the composite ring $B_{\mathrm{cris}}K \subset B_{\mathrm{dR}}$ is well-defined.
We denote by $\Phi_{\mathrm{cris}}$ the action of the absolute Frobenius on $B_{\mathrm{cris}}$
and define actions $\Phi_\tau$ ($\tau \in W_p$) on the ring $B_{\mathrm{cris}}K$ by
\begin{equation} \label{fronb}
\Phi_\tau:=\Phi_{\mathrm{cris}}^{\deg \tau}\otimes \tau.
\end{equation}

Let $K$, $\mathcal J_m$ be as in (\ref{pdr}): $K$ is a finite normal extension of $\mathbb Q_p$, $\mathcal J_m$ is a smooth model over $\mathcal O_K$ 
of the Jacobian variety $J_m$ of $F_m$. 
For simplicity, we fix embeddings $\overline{\mathbb Q} \hookrightarrow \mathbb C$, $\overline{\mathbb Q} \hookrightarrow \overline{\mathbb Q_p}$. 
The following results are well-known.
\begin{itemize}
\item There exists a canonical isomorphism between the singular cohomology group and the $p$-adic \'etale cohomology group:
\begin{equation} \label{ci1}
\mathrm H^1_{\mathrm{B}}(F_m(\mathbb C),\mathbb Q)\otimes_{\mathbb Q} \mathbb Q_p 
\cong \mathrm H^1_{p,\textrm{\'et}}(F_m\times_\mathbb Q \overline{\mathbb Q},\mathbb Q_p).
\end{equation}
\item There exists a canonical isomorphism between the $p$-adic \'etale cohomology group and the de Rham cohomology group:
\begin{equation} \label{ci2}
\mathrm H^1_{p,\textrm{\'et}}(F_m\times_\mathbb Q \overline{\mathbb Q},\mathbb Q_p)\otimes_{\mathbb Q_p} B_{\mathrm{cris}}K
\cong \mathrm{H_{dR}^1}(F_m,K)\otimes_K B_{\mathrm{cris}}K,
\end{equation}
which is compatible with the action of the Weil group $W_p$.
The element $\tau \in W_p$ acts on the left hand side by $1 \otimes \Phi_{\tau}$, 
on the right hand side by $\Phi_{\mathrm{cris}}^{\deg \tau} \otimes \Phi_{\tau}$.
It follows from the canonical isomorphism
\begin{equation*}
\begin{array}{ccc}
H^1_{p,\textrm{\'et}}(\mathcal J_m \times_{\mathcal O_K} \overline{\mathbb Q},\mathbb Q_p)\otimes_{\mathbb Q_p} B_{\mathrm{cris}} 
& \cong & \mathrm{H_{cris}^1}(\mathcal J_m\times_{\mathcal O_K} \mathbb F_K,W_K)\otimes_{W_K} B_{\mathrm{cris}} \\
1 \otimes \Phi_{\mathrm{cris}} & \leftrightarrow & \Phi_{\mathrm{cris}} \otimes \Phi_{\mathrm{cris}}
\end{array}
\end{equation*}
by identifying 
$\mathrm{H_{dR}^1}(F_m,K)=\mathrm{H_{cris}^1}(\mathcal J_m\times_{\mathcal O_K} \mathbb F_K,W_K)\otimes_{W_K} K$, 
$H^1_{p,\textrm{\'et}}(\mathcal J_m \times_{\mathcal O_K} \overline{\mathbb Q},\mathbb Q_p)
= \mathrm H^1_{p,\textrm{\'et}}(F_m\times_\mathbb Q \overline{\mathbb Q},\mathbb Q_p)$.
We note that the isomorphism (\ref{ci2}) is a restriction of the following canonical isomorphism 
\begin{equation} \label{ci4}
\mathrm H^1_{p,\textrm{\'et}}(F_m\times_\mathbb Q \overline{\mathbb Q},\mathbb Q_p)\otimes_{\mathbb Q_p} B_{\mathrm{dR}} 
\cong \mathrm H^1_{\mathrm{dR}}(F_m,\mathbb Q)\otimes_{\mathbb Q}B_{\mathrm{dR}}.
\end{equation}
\item The singular cohomology group is the dual of singular homology group. Namely, we have the non-degenerate pairing:
\begin{equation} \label{ci3}
\mathrm H_1(F_m(\mathbb C),\mathbb Q)\times \mathrm H^1_{\mathrm{B}}(F_m(\mathbb C),\mathbb Q) \rightarrow \mathbb Q.
\end{equation}
\end{itemize}
Combining (\ref{ci1}), (\ref{ci2}) and (\ref{ci3}), we get a period ring valued pairing 
\begin{equation} \label{pint}
\mathrm H_1(F_m(\mathbb C),\mathbb Q)\times \mathrm H^1_{\mathrm{dR}}(F_m,\mathbb Q) \rightarrow B_{\mathrm{cris}}K.
\end{equation}
We denote by $\int_{p,\gamma}\eta$ the image of $(\gamma,\eta)$ under the map (\ref{pint}).
This is a $p$-adic counterpart of 
the usual period $\int_{\gamma}\eta$, $\mathrm H_1(F_m(\mathbb C),\mathbb Q)\times \mathrm H^1_{\mathrm{dR}}(F_m,\mathbb Q) \rightarrow \mathbb C$.

\begin{thm} \label{main}
Let $p$ be a prime, $i,j,m$ integers satisfying $0<i,j<m$, $i+j\neq m$.
\begin{enumerate}
\item Assume that $(p,m)=1$. Then we have
\begin{equation*} 
\begin{split}
\Phi_\tau\left(\int_{p,\gamma} \eta_{\frac{i}{m},\frac{j}{m}}\right) 
= &\left (\prod_{k=1}^{\deg \tau} \frac{(-1)^{\varepsilon(\langle\frac{p^ki}{m}\rangle,\langle\frac{p^kj}{m}\rangle)}
p^{1-\varepsilon(\langle\frac{p^{k-1}i}{m}\rangle,\langle\frac{p^{k-1}j}{m}\rangle)}}
{B_p\langle\frac{p^ki}{m},\frac{p^kj}{m}\rangle}  \right) \\
& \times \frac{\langle \tau(\tfrac{i}{m})+\tau(\tfrac{j}{m}) \rangle^{\varepsilon(\tau(\frac{i}{m}),\tau(\frac{j}{m}))}}
{\langle \tfrac{i}{m}+\tfrac{j}{m} \rangle^{\varepsilon(\frac{i}{m},\frac{j}{m})}}
\int_{p,\gamma} \eta_{\tau(\frac{i}{m}),\tau(\frac{j}{m})}
\end{split}
\end{equation*}
for $\gamma \in \mathrm H_1(F_m(\mathbb C),\mathbb Q)$, $\tau\in W_p$ with $\deg \tau \geq 0$.
\item Assume that $p>2$, $p \mid m$, $(p,ij(i+j))=1$. Then we have
\begin{equation*}
\Phi_\tau\left(\frac{\int_{p,\gamma} \eta_{\frac{i}{m},\frac{j}{m}}}{B_p(\frac{i}{m},\frac{j}{m})}\right) 
\equiv p^{\frac{\deg\tau}{2}}
\frac{\int_{p,\gamma} \eta_{\tau(\frac{i}{m}),\tau(\frac{j}{m})}}{B_p(\tau(\frac{i}{m}),\tau(\frac{j}{m}))} \bmod \mu_{\infty}
\end{equation*}
for $\gamma \in \mathrm H_1(F_m(\mathbb C),\mathbb Q)$, $\tau\in W_p$.
\end{enumerate}
\end{thm}

\begin{proof}
Since the isomorphism (\ref{ci2}) is compatible with the action of $W_p$ 
and $W_p$ acts trivially on $\mathrm H^1_{\mathrm{B}}(F_m(\mathbb C),\mathbb Q)\otimes_{\mathbb Q} \mathbb Q_p 
\cong \mathrm H^1_{p,\textrm{\'et}}(F_m\times_\mathbb Q \overline{\mathbb Q},\mathbb Q_p)$, 
the assertion follows from Theorem \ref{mt1} immediately.
(Note that $\tau(\frac{i}{m})=\langle \frac{p^{\deg \tau}i}{m} \rangle$ in the former case, 
and that $\tau(\Gamma_p(\frac{i}{m})) \equiv \Gamma_p(\frac{i}{m}) \bmod \mu_\infty$ 
since $\Gamma_p(\frac{i}{m})^{p^N} \in \mathbb Q_p$ for large enough $N$ in the latter case.)
\end{proof}

\section{A reciprocity law on a $B_{\mathrm{cris}}\overline{\mathbb Q_p}$-valued beta function} \label{defofbeta} 

We will formulate a ``reciprocity law on a period ring-valued beta function'' which is a refinement of (\ref{rlofs}).
We first define $B_{\mathrm{cris}}\overline{\mathbb Q_p}$-valued beta function by
\begin{equation*}
\mathfrak B(\tfrac{i}{m},\tfrac{j}{m})
:=
\begin{cases}
\dfrac{B(\frac{i}{m},\frac{j}{m})}{\int_{\gamma} \eta_{\frac{i}{m},\frac{j}{m}}}\langle \tfrac{i}{m}+\tfrac{j}{m} \rangle^{\varepsilon(\frac{i}{m},\frac{j}{m})}
\int_{p,\gamma} \eta_{\frac{i}{m},\frac{j}{m}}& ((p,m)=1), \\
\dfrac{B(\frac{i}{m},\frac{j}{m})}{\int_{\gamma} \eta_{\frac{i}{m},\frac{j}{m}}}\dfrac{\int_{p,\gamma} \eta_{\frac{i}{m},\frac{j}{m}}}{B_p(\frac{i}{m},\frac{j}{m})} 
& (p\geq 2,\ p \mid m,\ (p,ij(i+j))=1).
\end{cases}
\end{equation*}
for $i,j \in \mathbb N$ with $i,j<m$, $i+j\neq m$ 
and $\gamma \in \mathrm H_1(F_m(\mathbb C),\mathbb Q)$ with $\int_{\gamma} \eta_{\frac{i}{m},\frac{j}{m}}\neq 0$.

\begin{lmm}
The value $\mathfrak B(\frac{i}{m},\frac{j}{m}) \in B_{\mathrm{cris}}\overline{\mathbb Q_p}$ 
does not depend on the choice of $m$, $\gamma $ whenever $\int_{\gamma} \eta_{\frac{i}{m},\frac{j}{m}}\neq 0$.
\end{lmm}

\begin{proof}
Put $G_m:=\mathbb Z/m\mathbb Z \oplus \mathbb Z/m\mathbb Z$. 
Then $H_1(F_m(\mathbb C),\mathbb Q)$ is the rank one $\mathbb Q[G_m]$-module in the sense of \cite[Proposition 4.9, (i)]{Ot}.
Let $\gamma_m \in \mathrm H_1(F_m(\mathbb C),\mathbb Q)$ be the generator.
Then we can write any element in $\mathrm H_1(F_m(\mathbb C),\mathbb Q)$ as $\rho \gamma_m$ with $\rho \in \mathbb Q[G_m]$. 
Hence the dependences of $\int_{\gamma}\eta_{\frac{i}{m},\frac{j}{m}}$, $\int_{p,\gamma}\eta_{\frac{i}{m},\frac{j}{m}}$ on $\gamma$ are canceled out
since $\eta_{\frac{i}{m},\frac{j}{m}}$ is a simultaneous eigenvector for any $g \in G_m$.
In order to see the dependence on $m$, we consider the canonical map $f \colon F_{nm} \rightarrow F_m$, $(x,y) \mapsto (x^n,y^n)$ ($2\leq n \in \mathbb N$).
Then we have $f^* \eta_{\frac{i}{m},\frac{j}{m}} = n\eta_{\frac{ni}{nm},\frac{nj}{nm}}$, so the assertion follows.
\end{proof}

By (\ref{pisb2}), we can take $\gamma_0$ so that $B(\frac{i}{m},\frac{j}{m})=\int_{\gamma_0} \eta_{\frac{i}{m},\frac{j}{m}}$.
Then we have
$\mathfrak B(\frac{i}{m},\frac{j}{m})=\langle \tfrac{i}{m}+\tfrac{j}{m} \rangle^{\varepsilon(\frac{i}{m},\frac{j}{m})}\int_{p,\gamma_0} \eta_{\frac{i}{m},\frac{j}{m}}$  
(resp.\ $\frac{\int_{p,\gamma_0} \eta_{\frac{i}{m},\frac{j}{m}}}{B_p(\frac{i}{m},\frac{j}{m})}$) if $(p,m)=1$ (resp.\ $p \mid m$).
Therefore Theorem \ref{main} implies the following reciprocity law on our beta function.

\begin{thm} \label{main3}
Let $p$ be a prime, $i,j,m$ integers satisfying $0<i,j<m$, $i+j\neq m$, and $\tau\in W_p$.
\begin{enumerate}
\item Assume that $(p,m)=1$, $\deg \tau \geq 0$. Then we have
\begin{equation*}
\Phi_\tau(\mathfrak B(\tfrac{i}{m},\tfrac{j}{m})) 
= \left (\prod_{k=1}^{\deg \tau} \frac{(-1)^{\varepsilon(\langle\frac{p^ki}{m}\rangle,\langle\frac{p^kj}{m}\rangle)}
p^{1-\varepsilon(\langle\frac{p^{k-1}i}{m}\rangle,\langle\frac{p^{k-1}j}{m}\rangle)}}
{B_p\langle\frac{p^ki}{m},\frac{p^kj}{m}\rangle}  \right)  \mathfrak B(\tau(\tfrac{i}{m}),\tau(\tfrac{j}{m})).
\end{equation*}
\item Assume that $p>2$, $p \mid m$, $(p,ij(i+j))=1$. Then we have
\begin{equation*}
\Phi_\tau(\mathfrak B(\tfrac{i}{m},\tfrac{j}{m})) \equiv p^{\frac{\deg \tau}{2}} \mathfrak B(\tau(\tfrac{i}{m}),\tau(\tfrac{j}{m})) \bmod \mu_{\infty}.
\end{equation*}
\end{enumerate}
\end{thm}

We prove the key formula (\ref{bpvsbc}) which states that 
the product $\mathfrak B(\frac{i}{m},\frac{j}{m})\mathfrak B(\frac{m-i}{m},\frac{m-j}{m})$ of our beta function is essentially equal to 
the product $B(\frac{i}{m},\frac{j}{m})B(\frac{m-i}{m},\frac{m-j}{m})$ of the classical beta function.
To do this, we prepare some Lemmas.
 
\begin{lmm} \label{btog2}
\begin{enumerate}
\item Assume that $(p,m)=1$. Then we have 
\begin{equation*}
B_p \langle \tfrac{i}{m},\tfrac{j}{m} \rangle B_p \langle \tfrac{m-i}{m},\tfrac{m-j}{m} \rangle =\pm 1.
\end{equation*}
\item Assume $p>2$, $p \mid m$, $(p,ij(i+j))=1$. Then we have 
\begin{equation*}
B_p(\tfrac{i}{m},\tfrac{j}{m})B_p(\tfrac{m-i}{m},\tfrac{m-j}{m}) \equiv (\tfrac{m}{m-i-j})^* \bmod \mu_\infty.
\end{equation*}
\end{enumerate}
\end{lmm}

\begin{proof}
The first assertion follows from (\ref{perf}).
In the case of $\frac{i}{m},\frac{j}{m} \in \mathbb Q_p-\mathbb Z_p$, we have
\begin{equation*}
B_p(\tfrac{i}{m},\tfrac{j}{m})B_p(\tfrac{m-i}{m},\tfrac{m-j}{m})
= \frac{\Gamma_p(\frac{i}{m})\Gamma_p(\frac{j}{m})\Gamma_p(\frac{m-i}{m})\Gamma_p(\frac{m-j}{m})}{\Gamma_p(\frac{i+j}{m})\Gamma_p(\frac{m-i-j}{m})}
(\tfrac{m}{m-i-j})^*.
\end{equation*}
by (\ref{feqofpg}).
Therefore it suffices to show that $\Gamma_p(\alpha)\Gamma_p(1-\alpha) \equiv 1 \bmod \mu_\infty$ for $\alpha \in \mathbb Q_p-\mathbb Z_p$, which is 
a generalization of (\ref{perf}).
It reduces to the formula
\begin{equation*}
L\Gamma_p(a,(a_1))+L\Gamma_p(a_1-a,(a_1))=0
\end{equation*}
for $a \in \mathbb Z_p^\times$, $a_1 \in p\mathbb Z_p$ by (\ref{defofgp}).
This formula follows from a property of Bernoulli polynomials: $B_n(1-x)=(-1)^nB_n(x)$ ($0\leq n \in \mathbb Z$).
In fact, by (\ref{lgp}) we can write  
\begin{equation*}
\begin{split}
&L\Gamma_{p,1}(a,(a_1))+L\Gamma_p(a_1-a,(a_1)) \\
&=-\left(B_1(\tfrac{a}{a_1})+B_1(1-\tfrac{a}{a_1})\right) \\
&\quad -\sum_{k=1}^\infty \frac{(-1)^k}{k}
\sum_{l=0}^k\frac{k!}{l!(k-l)!}a_1^l\left(\omega(a)^{-l}B_{l+1}(\tfrac{a}{a_1})+(\omega(a_1-a))^{-l}B_{l+1}(1-\tfrac{a}{a_1})\right)(-1)^{k-l}.
\end{split}
\end{equation*}
Since $\omega(a_1-a)=-\omega(a)$, we have $\omega(a)^{-l}B_{l+1}(\frac{a}{a_1})+(\omega(a_1-a))^{-l}B_{l+1}(1-\frac{a}{a_1})=0$.
Then the assertion is clear.
\end{proof}

\begin{lmm} \label{btog3}
The ratio $\frac{\int_{p,\gamma_1} \eta_{\frac{i}{m},\frac{j}{m}}\int_{p,\gamma_2} \eta_{\frac{m-i}{m},\frac{m-j}{m}}}
{{\int_{\gamma_1} \eta_{\frac{i}{m},\frac{j}{m}}}\int_{\gamma_2} \eta_{\frac{m-i}{m},\frac{m-j}{m}}}$ 
does not depend on $\gamma_1,\gamma_2,i,j,m$ whenever 
$\int_{\gamma_1} \eta_{\frac{i}{m},\frac{j}{m}},\int_{\gamma_2} \eta_{\frac{m-i}{m},\frac{m-j}{m}}\neq 0$.
To be precise, there exists a constant $(2\pi i)_p \in B_{\mathrm{dR}}^\times$, which is a $p$-adic counterpart of $2\pi i$, satisfying 
\begin{equation*}
\frac{\int_{\gamma_1} \eta_{\frac{i}{m},\frac{j}{m}}\int_{\gamma_2} \eta_{\frac{m-i}{m},\frac{m-j}{m}}}{2\pi i}=
\frac{\int_{p,\gamma_1} \eta_{\frac{i}{m},\frac{j}{m}}\int_{p,\gamma_2} \eta_{\frac{m-i}{m},\frac{m-j}{m}}}{(2\pi i)_p} \in \overline{\mathbb Q}
\end{equation*}
for all $i,j,m$ with $0<i,j<m$, $i+j\neq m$ and for all $\gamma_1,\gamma_2 \in \mathrm H_1(F_m(\mathbb C),\mathbb Q)$.
Here we define $(2\pi i)_p$ as follows. Taking the projective line $\mathbb P^1$ ($=F_1: x+y=1$) 
and the basis $c_b \in \mathrm{H}_2(\mathbb P^1,\mathbb Q)$, 
$c_{dr} \in \mathrm{H}^2_{\mathrm{dR}}(\mathbb P^1,\mathbb Q)$, we put
\begin{equation*}
(2\pi i)_p:=\frac{(2\pi i) \int_{p,c_b} c_{dr}}{\int_{c_b} c_{dr}},
\end{equation*}
where $\int \colon \mathrm{H}_2(\mathbb P^1,\mathbb Q) \times \mathrm{H}^2_{\mathrm{dR}}(\mathbb P^1,\mathbb Q) \rightarrow \mathbb C$, 
$\int_p \colon \mathrm{H}_2(\mathbb P^1,\mathbb Q) \times \mathrm{H}^2_{\mathrm{dR}}(\mathbb P^1,\mathbb Q) \rightarrow B_{\mathrm{cris}}$ 
are given similarly to (\ref{pint}).
In particular, we see that $(2\pi i)_p \in B_{\mathrm{cris}}$ and that $\Phi_{\mathrm{cris}}((2\pi i)_p)=p(2\pi i)_p$ 
since $H^2_{p,\textrm{\'et}}(\mathbb P^1\times_\mathbb Q \overline{\mathbb Q},\mathbb Q_p) \cong 
H^1_{p,\textrm{\'et}}(\mathbb G_m\times_\mathbb Q \overline{\mathbb Q},\mathbb Q_p) \cong$ the Lefschetz motive $\mathbb Q_p(-1)$.
\end{lmm}

\begin{proof}
Let $\gamma_1,\gamma_2 \in \mathrm H_1(F_m(\mathbb C),\mathbb Q)$ satisfy $\int_{\gamma_1} \eta_{\frac{i}{m},\frac{j}{m}},
\int_{\gamma_2} \eta_{\frac{m-i}{m},\frac{m-j}{m}}\neq 0$. 
First we note that we can write via the de Rham isomorphism
\begin{equation*}
\begin{array}{ccc}
\mathrm H^1_{\mathrm{B}}(F_m(\mathbb C),\mathbb Q(\zeta_m))\otimes_{\mathbb Q(\zeta_m)} \mathbb C 
& \cong & \mathrm{H_{dR}^1}(F_m,\mathbb Q)\otimes_{\mathbb Q} \mathbb C \\
\gamma_1^*\otimes \int_{\gamma_1} \eta_{\frac{i}{m},\frac{j}{m}} & \leftrightarrow & \eta_{\frac{i}{m},\frac{j}{m}} \otimes 1
\end{array}
\end{equation*}
with an element $\gamma_1^* \in \mathrm H^1_{\mathrm{B}}(F_m(\mathbb C),\mathbb Q(\zeta_m))$.
(By the complex multiplication, we can take a tentative element $\gamma_1^*$ 
so that $\gamma_1^* \otimes c \leftrightarrow \eta_{\frac{i}{m},\frac{j}{m}} \otimes 1$ with $c \in \mathbb C^\times$.
Then we normalize it as $\frac{\gamma_1^*}{(\gamma_1,\gamma_1^*)}$ by using the paring $(\gamma_1,\gamma_1^*)$ of (\ref{ci3}).)
We take $\gamma_2^*$ in the same manner.
Since the de Rham isomorphism is a ring isomorphism under cup products, we have
\begin{equation*}
\begin{array}{ccc}
\mathrm H_B^2(F_m(\mathbb C),\mathbb Q(\zeta_m))\otimes_{\mathbb Q(\zeta_m)} \mathbb C 
& \cong & \mathrm{H_{dR}^2}(F_m,\mathbb Q)\otimes_{\mathbb Q} \mathbb C \\
(\gamma_1^* \cup \gamma_2^*) \otimes (\int_{\gamma_1} \eta_{\frac{i}{m},\frac{j}{m}}\int_{\gamma_2} \eta_{\frac{m-i}{m},\frac{m-j}{m}}) & \leftrightarrow & 
(\eta_{\frac{i}{m},\frac{j}{m}} \cup \eta_{\frac{m-i}{m},\frac{m-j}{m}}) \otimes 1.
\end{array}
\end{equation*}
Similarly we have
\begin{equation*}
\begin{array}{ccc}
\mathrm H_B^2(F_m(\mathbb C),\mathbb Q(\zeta_m))\otimes_{\mathbb Q(\zeta_m)} B_{\mathrm{dR}} 
& \cong & \mathrm{H_{dR}^2}(F_m,\mathbb Q)\otimes_{\mathbb Q} B_{\mathrm{dR}} \\
(\gamma_1^* \cup \gamma_2^*) \otimes (\int_{p,\gamma_1} \eta_{\frac{i}{m},\frac{j}{m}}\int_{p,\gamma_2} \eta_{\frac{m-i}{m},\frac{m-j}{m}}) & \leftrightarrow & 
(\eta_{\frac{i}{m},\frac{j}{m}} \cup \eta_{\frac{m-i}{m},\frac{m-j}{m}}) \otimes 1,
\end{array}
\end{equation*}
since (\ref{ci1}), (\ref{ci4}) are also ring isomorphisms.
We note that $\dim_{\mathbb Q}\mathrm{H_{dR}^2}(F_m,\mathbb Q)=1$ and that 
the canonical map $\iota_m \colon F_m \rightarrow F_1$, $(x,y) \mapsto (x^m,y^m)$ induces 
isomorphisms $\mathrm H_*^2(\mathbb P^1,\mathbb Q) \cong \mathrm H_*^2(F_m,\mathbb Q)$ ($*=\mathrm{B},\mathrm{dR}$).
Hence it suffices to show that 
the cup product $\eta_{\frac{i}{m},\frac{j}{m}} \cup \eta_{\frac{m-i}{m},\frac{m-j}{m}} \neq 0$.
In the following of this proof, we compute $\eta_{\frac{i}{m},\frac{j}{m}} \cup \eta_{\frac{m-i}{m},\frac{m-j}{m}}$ directly by using \v{C}ech cohomology.
We take the covering $U_1^{(m)},U_2^{(m)}$ of the projective Fermat curve $F_m: X^m+Y^m=Z^m$ so that they corresponds to $Z\neq 0$, $Y\neq 0$ respectively.
Namely, $U_1^{(m)}$ (resp.\ $U_2^{(m)}$) is the affine curve defined by $x^m+y^m=1$ with $x=X/Z$, $y=Y/Z$ (resp.\ $x'^m+1=z'^m$ with $x'=X/Y$, $z'=Z/Y$).
Then we can write
\begin{equation*}
\begin{split}
&\mathrm{H_{dR}^1}(F_m,\mathbb Q)
=\frac{\{(\omega_1,\omega_2,f) \mid \omega_i \in \Omega^1(U_i^{(m)}),f \in \Omega^0(U_1^{(m)}\cap U_2^{(m)}),\omega_1-\omega_2=df\}}
{\{(df_1,df_2,f_1-f_2) \mid f_i \in \Omega^0(U_i^{(m)})\}}, \\
&\mathrm{H_{dR}^2}(F_m,\mathbb Q)
=\frac{\Omega^1(U_1^{(m)}\cap U_2^{(m)})}
{\{\omega_1-\omega_2-df \mid \omega_i \in \Omega^1(U_i^{(m)}),f \in \Omega^0(U_1^{(m)}\cap U_2^{(m)})\}}, \\
&(\omega_1,\omega_2,f) \cup (\omega_1',\omega_2',f')=-f'\omega_1+f \omega_2', \\
&\eta_{\frac{i}{m},\frac{j}{m}} =
\begin{cases}
(x^{i-1}y^{j-m}dx,x'^{i-m}z'^{m-i-j-1}dz',0) & \text{ if } i+j<m, \\
(x^{i-1}y^{j-m}dx,\frac{j-m}{i+j-m}x'^{i-m}z'^{2m-i-j-1}dz',\frac{1}{i+j-m}x^{i}y^{j-m}) & \text{ if } i+j>m. \\
\end{cases}
\end{split}
\end{equation*}
Therefore we get for $i,j$ with $i+j<m$
\begin{equation*}
\eta_{\frac{i}{m},\frac{j}{m}} \cup \eta_{\frac{m-i}{m},\frac{m-j}{m}}=\tfrac{x^{m}y^{-m}}{i+j-m}\tfrac{dx}{x}=
\iota_m^*(\tfrac{1}{m(i+j-m)}\tfrac{dx}{1-x}).
\end{equation*}
We easily see that $\tfrac{1}{m(i+j-m)}\tfrac{dx}{1-x}$ is non-zero in 
\begin{equation*}
\begin{split}
\mathrm{H_{dR}^2}(F_1,\mathbb Q)
&=\frac{\Omega^1(U_1^{(1)}\cap U_2^{(1)})}
{\{\omega_1-\omega_2-df \mid \omega_i \in \Omega^1(U_i^{(1)}),f \in \Omega^0(U_1^{(1)}\cap U_2^{(1)})\}} \\
&=\frac{\mathbb Q[x,\frac{1}{x-1}]dx}{\mathbb Q[x]dx+\mathbb Q[\frac{1}{x-1}]\frac{dx}{(x-1)^2}+d(\mathbb Q[x,\frac{1}{x-1}])}
\end{split}
\end{equation*}
since $\mathrm{Res}_{x=1}\frac{dx}{1-x}\neq 0$.
Hence the assertion is clear.
\end{proof}

By Lemma \ref{sisprofb} and (\ref{pisb}), 
the monomial relations $\frac{\int_{\gamma_1} \eta_{\frac{i}{m},\frac{j}{m}}\int_{\gamma_2} \eta_{\frac{m-i}{m},\frac{m-j}{m}}}{2\pi i} \in \overline{\mathbb Q}$ 
in Lemma \ref{btog3} imply the algebraicity of Stark units over the rational number field. Namely, we have the following Corollary.

\begin{crl} \label{crl1}
{\rm (Alg$'$)} for $F=\mathbb Q$ holds true. That is, we have
\begin{equation*}
u_\mathbb Q(\sigma) \in \overline{\mathbb Q} \qquad (\sigma \in \mathrm{Gal}(\mathbb Q(\zeta_m)^+/\mathbb Q)).
\end{equation*}
\end{crl}

By Lemma \ref{btog2}-\rm{(ii)},Lemma \ref{btog3} and $\langle \tfrac{i}{m}+\tfrac{j}{m} \rangle^{\varepsilon(\frac{i}{m},\frac{j}{m})}
\langle \tfrac{m-i}{m}+\tfrac{m-j}{m} \rangle^{\varepsilon(\frac{m-i}{m},\frac{m-j}{m})}=\pm(\tfrac{i}{m}+\tfrac{j}{m}-1)$, we can write
\begin{equation} \label{bpvsbc}
\frac{\mathfrak B(\tfrac{i}{m},\tfrac{j}{m})\mathfrak B(\tfrac{m-i}{m},\tfrac{m-j}{m})}{(2 \pi i)_p} \\
\equiv \left(1-\tfrac{i}{m}-\tfrac{j}{m}\right)^* \frac{B(\frac{i}{m},\frac{j}{m})B(\frac{m-i}{m},\frac{m-j}{m})}{2 \pi i} \bmod \mu_\infty
\end{equation}
in both cases: $(p, m)=1$ or $p>2$, $p \mid m$, $(p,ij(i+j))=1$.  
Therefore Lemma \ref{btog3} and (\ref{pisb}) imply not only $B(\tfrac{i}{m},\tfrac{j}{m})B(\tfrac{m-i}{m},\tfrac{m-j}{m})\pi^{-1} \in \overline{\mathbb Q}$,
but also $\mathfrak B(\tfrac{i}{m},\tfrac{j}{m})\mathfrak B(\tfrac{m-i}{m},\tfrac{m-j}{m})(2 \pi i)_p^{-1} \in \overline{\mathbb Q}$.

Finally in this section, we derive (Rec$'$) $\bmod \mu_\infty$ in the case of $F=\mathbb Q$ from Theorem \ref{main3}.
By Lemma \ref{btog2}-\rm{(i)} and $\varepsilon(\frac{i}{m},\frac{j}{m})+\varepsilon(\frac{m-i}{m},\frac{m-j}{m})=1$,
we see that Theorem \ref{main3} implies
\begin{equation*}
\Phi_\tau(\mathfrak B(\tfrac{i}{m},\tfrac{j}{m})\mathfrak B(\tfrac{m-i}{m},\tfrac{m-j}{m})) 
\equiv p^{\deg \tau} \mathfrak B(\tau(\tfrac{i}{m}),\tau(\tfrac{j}{m}))\mathfrak B(\tau(\tfrac{m-i}{m}),\tau(\tfrac{m-j}{m})) \bmod \mu_{\infty}
\quad (\tau \in W_p)
\end{equation*}
in both cases.
Besides, we have (\ref{bpvsbc}), $\Phi_{\mathrm{cris}}((2\pi i)_p)=p(2\pi i)_p$, and $\Phi_\tau$ is $\tau$-semi linear, so we get
\begin{equation*}
\begin{split}
&\tau\left(\left(1-\tfrac{i}{m}-\tfrac{j}{m}\right)^*\frac{ B(\frac{i}{m},\frac{j}{m})B(\frac{m-i}{m},\frac{m-j}{m})}{\pi} \right) \\
&\equiv \left(1-\tau(\tfrac{i}{m})-\tau(\tfrac{j}{m})\right)^*\frac{ B(\tau(\frac{i}{m}),\tau(\frac{j}{m}))B(\tau(\frac{m-i}{m}),\tau(\frac{m-j}{m}))}{\pi} 
\bmod \mu_{\infty} \quad (\tau \in W_p).
\end{split}
\end{equation*}
Hence we obtain the desired formula $\tau (u_\mathbb Q(\sigma)) \equiv u_\mathbb Q(\tau \sigma) \bmod \mu_\infty$ 
for $\tau \in W_p$ by (\ref{suoverq}) and Lemma \ref{sisprofb}.
Since we can vary $p$ and the embedding $\overline{\mathbb Q} \hookrightarrow \overline{\mathbb Q_p}$, the same formula is valid 
for any $\tau \in \mathrm{Gal}(\overline{\mathbb Q}/\mathbb Q)$.
Namely, 

\begin{crl} \label{crl2}
{\rm (Rec$'$)} for $F=\mathbb Q$ holds true up to multiplication by a root of unity. That is, we have
\begin{equation*}
\tau (u_\mathbb Q(\sigma)) \equiv u_\mathbb Q(\tau \sigma) \bmod \mu_\infty \qquad 
(\sigma \in \mathrm{Gal}(\mathbb Q(\zeta_m)^+/\mathbb Q),\tau \in  \mathrm{Gal}(\overline{\mathbb Q}/\mathbb Q)).
\end{equation*}
\end{crl}

\end{document}